\documentclass[12pt]{amsart}
\usepackage{amssymb}

\newtheorem{theorem}{Theorem}[section]
\newtheorem{lemma}{Lemma}[section]

\title[Arithmetic progressions of squares, cubes and $n$-th powers]
{Arithmetic progressions of squares, cubes and $n$-th powers}

\author{$\mbox{L. Hajdu}^1$, $\mbox{Sz. Tengely}^2$}

\thanks{1) Research supported in part by the J\'anos Bolyai Research
Fellowship of the Hungarian Academy of Sciences, by the OTKA grants
T48791 and T67580, and by the National Office for Research and
Technology.
\hfill\break\indent
2) Research supported in part by the Magyary Zolt\'an Higher
Educational Public Foundation}

\subjclass[2000]{11D61, 11Y50}

\keywords{perfect powers, arithmetic progressions}

\address{Institute of Mathematics\newline
 \indent University of Debrecen\newline
 \indent and the Number Theory Research Group\newline
 \indent of the Hungarian Academy of Sciences\newline
 \indent P.O.Box 12\newline
 \indent 4010 Debrecen\newline
 \indent Hungary}

\email{hajdul@math.klte.hu, tengely@math.klte.hu}

\begin{document}

\begin{abstract}
In this paper we continue the investigations about unlike powers in
arithmetic progression. We provide sharp upper bounds for the length
of primitive non-constant arithmetic progressions consisting of
squares/cubes and $n$-th powers.
\end{abstract}

\maketitle

\section{Introduction}
It was claimed by Fermat and proved by Euler (see \cite{Dickson} pp.
440 and 635) that four distinct squares cannot form an arithmetic
progression. It was shown by Darmon and Merel \cite{intDM} that,
apart from trivial cases, there do not exist three-term arithmetic
progressions consisting of $n$-th powers, provided $n\geq 3$. An
arithmetic progression $a_1,a_2,\hdots,a_t$ of integers is called
primitive if $\gcd(a_1,a_2)=1$. A recent result of Hajdu \cite{Lajos}
implies that if
\begin{equation}
\label{mainap}
x_1^{l_1},\hdots,x_t^{l_t}
\end{equation}
is a primitive arithmetic progression in $\mathbb{Z}$ with
$2\leq l_i\leq L$ $(i=1,\hdots,t)$, then $t$ is bounded by some
constant $c(L)$ depending only on $L$. Note that $c(L)$ is effective,
but it is not explicitly given in \cite{Lajos}, and it is a very
rapidly growing function of $L$.

An the other hand, it is known (see e.g. \cite{Mo}, \cite{DaGr},
\cite{T} and the references given there) that there exist exponents
$l_1,l_2,l_3\geq 2$ for which there are infinitely many primitive
arithmetic progressions of the form \eqref{mainap}. In this case the
exponents in question satisfy the condition
$$
\frac{1}{l_1}+\frac{1}{l_2}+\frac{1}{l_3}\geq 1.
$$
In \cite{BGyHT} Bruin, Gy\H{o}ry, Hajdu and Tengely among other things
proved that for any $t\geq 4$ and $L\geq 3$ there are only finitely
many primitive arithmetic progressions of the form \eqref{mainap} with
$2\leq l_i\leq L$ $(i=1,\ldots,t)$. Furthermore, they showed that in
case of $L=3$ we have $x_i=\pm 1$ for all $i=1,\ldots,t$.

The purpose of the present paper is to give a good, explicit upper bound
for the length $t$ of the progression \eqref{mainap} under certain
restrictions. More precisely, we consider the cases when the set of
exponents is given by $\{2,n\}$, $\{2,5\}$ and $\{3,n\}$, and (excluding
the trivial cases) we show that the length of the progression is at most
six, four and four, respectively.

\section{Results}

\begin{theorem}
\label{thm2n}
Let $n$ be a prime and $x_1^{l_1},\hdots,x_t^{l_t}$ be a primitive
non-constant arithmetic progression in $\mathbb{Z}$ with $l_i\in\{2,n\}$
$(i=1,\hdots,t)$. Then we have $t\leq 6$. Further, if $t=6$ then
\begin{equation*}
(l_1,l_2,l_3,l_4,l_5,l_6)=(2,n,n,2,2,2),(2,2,2,n,n,2).
\end{equation*}
\end{theorem}

In the special case $n=5$ we are able to prove a sharper result.

\begin{theorem}
\label{thm25}
Let $x_1^{l_1},\hdots,x_t^{l_t}$ be a primitive non-constant arithmetic
progression in $\mathbb{Z}$ with $l_i\in\{2,5\}$ $(i=1,\hdots,t)$. Then
we have $t\leq 4$. Further, if $t=4$ then
$$
(l_1,l_2,l_3,l_4)=(2,2,2,5),(5,2,2,2).
$$
\end{theorem}

\begin{theorem}
\label{thm3n}
Let $n$ be a prime and $x_1^{l_1},\hdots,x_t^{l_t}$ be a primitive
non-constant arithmetic progression in $\mathbb{Z}$ with $l_i\in\{3,n\}$
$(i=1,\hdots,t)$. Then we have $t\leq 4$. Further, if $t=4$ then
$$
(l_1,l_2,l_3,l_4)=(3,3,n,n),(n,n,3,3),(3,n,n,3),(n,3,3,n).
$$
\end{theorem}

Note that Theorems \ref{thm25} and \ref{thm3n} are almost best possible.
This is demonstrated by the primitive non-constant progression $-1,0,1$.
(In fact one can easily give infinitely many examples of arithmetic
progressions of length three, consisting of squares and fifth powers.)

We also remark that by a previously mentioned result from \cite{BGyHT},
the number of progressions of length at least four is finite in each
case occurring in the above theorems.

\section{Proofs of Theorems \ref{thm2n} and \ref{thm3n}}

In the proof of these theorems we need several results about ternary
equations of signatures $(n,n,2)$ and $(n,n,3)$, respectively. We start
this section with summarizing these statements. The first three lemmas
are known from the literature, while the fourth one is new.

\begin{lemma}
\label{nn2}
Let $n$ be a prime. Then the Diophantine equations
\begin{gather*}
X^n+Y^n=2Z^2\ \ \ (n\geq 5),\\
X^n+Y^n=3Z^2\ \ \ (n\geq 5),\\
X^n+4Y^n=3Z^2\ \ \ (n\geq 7)
\end{gather*}
have no solutions in nonzero pairwise coprime integers $(X,Y,Z)$
with $XY\neq\pm 1$.
\end{lemma}

\begin{proof} The statement follows from results of Bennett and
Skinner \cite{BeSk}, and Bruin \cite{Nilsnn2}.
\end{proof}

\begin{lemma}
\label{nn3}
Let $n\geq 5$ be a prime. Then the Diophantine equation
$$
X^n+Y^n=2Z^3
$$
has no solutions in coprime nonzero integers $X,Y,Z$ with $XYZ\neq\pm 1$.
\end{lemma}

\begin{proof} The result is due to Bennett, Vatsal and Yazdani
\cite{BeVaYa}.
\end{proof}

\begin{lemma}
\label{nnn}
Let $n\geq 3$ be a prime. Then the Diophantine equation
$$
X^n+Y^n=2Z^n
$$
has no solutions in coprime nonzero integers $X,Y,Z$ with $XYZ\neq\pm 1$.
\end{lemma}

\begin{proof} The result is due to Darmon and Merel \cite{intDM}.
\end{proof}

\begin{lemma}
\label{33n}
Let $n\geq 3$ be a prime. Then the Diophantine equation
$$
X^3+Y^3=2Z^n
$$
has no solutions in coprime nonzero integers $X,Y,Z$ with $XYZ\neq\pm 1$
and $3\nmid Z$.
\end{lemma}

\begin{proof} First note that in case of $n=3$ the statement follows
from Lemma \ref{nnn}. Let $n\geq 5$, and assume to the contrary that
$(X,Y,Z)$ is a solution
to the equation with $\gcd(X,Y,Z)=1$, $XYZ\neq \pm 1$ and $3\nmid Z$.
Note that the coprimality of $X,Y,Z$ shows that $XY$ is odd. We have
$$
(X+Y)(X^2-XY+Y^2)=2Z^n.
$$
Our assumptions imply that $\gcd(X+Y,X^2-XY+Y^2)\mid 3$, whence
$2\nmid XY$ and $3\nmid Z$ yield that
$$
X+Y=2U^n \text{ and } X^2-XY+Y^2=V^n
$$
hold, where $U,V\in\mathbb{Z}$ with $\gcd(U,V)=1$. Combining these
equations we get
$$
f(X):=3X^2-6U^nX+4U^{2n}-V^n=0.
$$ 
Clearly, the discriminant of $f$ has to be a square in $\mathbb{Z}$,
which leads to an equality of the form
$$  
V^n-U^{2n}=3W^2
$$
with some $W\in\mathbb{Z}$. However, this is impossible by Lemma
\ref{nn2}.
\end{proof}

Now we are ready to prove our Theorems \ref{thm2n} and \ref{thm3n}.

\begin{proof}[Proof of Theorem \ref{thm2n}] Suppose that we have an
arithmetic progression \eqref{mainap} of the desired form, with $t=6$.
In view of a result from \cite{BGyHT} about the case $n=3$ and Theorem
\ref{thm25}, without loss of generality we may assume that $n\geq 7$.

First note that the already mentioned classical result of Fermat and
Euler implies that we cannot have four consecutive squares in our
progression. Further, observe that Lemmas \ref{nn2} and \ref{nnn} imply
that we cannot have three consecutive terms with exponents $(n,2,n)$
and $(n,n,n)$, respectively, and further that $(l_1,l_3,l_5)=(n,2,n)$,
$(n,n,n)$ are also impossible.

If $(l_1,l_2,l_3,l_4,l_5)=(n,2,2,n,2)$ or $(2,n,2,2,n)$, then we have
$$
4x_4^n-x_1^n=3x_5^2\ \ \ \mbox{or}\ \ \ 4x_2^n-x_5^n=3x_1^2,
$$
respectively, both equations yielding a contradiction by Lemma
\ref{nn2}.

To handle the remaining cases, let $d$ denote the common difference
of the progression. Let $(l_1,l_2,l_3,l_4,l_5)=(2,2,n,2,2)$. Then (as
clearly $x_1\neq 0$) we have
$$
(1+X)(1+3X)(1+4X)=Y^2
$$
where $X=d/x_1$ and $Y=x_2 x_4 x_5 / x_1$. However, a simple calculation
with Magma \cite{MAGMA} shows that the rank of this elliptic curve is zero,
and it has exactly eight torsion points. However, none of these torsion
points gives rise to any appropriate arithmetic progression.

When $(l_1,l_2,l_3,l_4,l_5,l_6)=(2,2,n,n,2,2)$, then in a similar manner
we get
$$
(1+X)(1+4X)(1+5X)=Y^2
$$
with $X=d/x_1$ and $Y=x_2 x_5 x_6 / x_1$, and just as above, we get a
contradiction.

In view of the above considerations, a simple case-by-case analysis yields
that the only remaining possibilities are the ones listed in the theorem.
Hence to complete the proof we need only to show that the possible
six-term progressions cannot be extended to seven-term ones. Using
symmetry it is sufficient to deal with the case given by
$$
(l_1,l_2,l_3,l_4,l_5,l_6)=(2,n,n,2,2,2).
$$
However, one can easily verify that all the possible extensions lead to
a case treated before, and the theorem follows.
\end{proof}

\begin{proof}[Proof of Theorem \ref{thm3n}] In view of Lemma \ref{nnn}
and the previously mentioned result from \cite{BGyHT} we may suppose
that $n\geq 5$. Assume that we have an arithmetic progression of the
indicated form, with $t=4$. By the help of Lemmas \ref{nn3} and \ref{nnn}
we get that there cannot be three consecutive terms with exponents
$(n,3,n)$, and $(3,3,3)$ or $(n,n,n)$, respectively. Hence a simple
calculation yields that the only possibilities (except for the ones
listed in the theorem) are given by
$$
(l_1,l_2,l_3,l_4)=(3,n,3,3),(3,3,n,3).
$$
Then Lemma \ref{33n} yields that $3\mid x_2$ and $3\mid x_3$,
respectively. However, looking at the progressions modulo $9$ and using
that $x^3\equiv 0,\pm 1\pmod{9}$ for all $x\in{\mathbb Z}$ we get a
contradiction with the primitivity condition in both cases.

Finally, one can easily check that the extensions of the four-term
sequences corresponding to the exponents listed in the statement to
five-term ones, yield cases which have been treated already. Hence the
proof of the theorem is complete.
\end{proof}

\section{Proof of Theorem \ref{thm25}}

To prove this theorem we need some lemmas, obtained by the help of
elliptic Chabauty's method.

\begin{lemma}
\label{l25a}
Let $\alpha=\sqrt[5]{2}$ and put $K=\mathbb{Q}(\alpha)$. Then the
equations
\begin{equation}
\label{e25a1}
C_1:\quad\alpha^4X^4+\alpha^3X^3+\alpha^2X^2+\alpha X+1=(\alpha-1)Y^2
\end{equation}
and
\begin{equation}
\label{e25a2}
C_2:\quad\alpha^4X^4-\alpha^3 X^3+\alpha^2X^2-\alpha X+1=
(\alpha^4-\alpha^3+\alpha^2-\alpha+1)Y^2
\end{equation}
in $X\in\mathbb{Q}$, $Y\in K$ have the only solutions
$$
(X,Y)=(1,\pm(\alpha^4+\alpha^3+\alpha^2+\alpha+1)),\
\left(-\frac{1}{3},\pm\frac{3\alpha^4+5\alpha^3-\alpha^2+3\alpha+5}{9}\right)
$$
and $(X,Y)=(1,\pm 1)$, respectively.
\end{lemma}

\begin{proof} 
Using the so-called elliptic Chabauty's method (see \cite{NB1},
\cite{NB2}) we determine all points on the above curves for which $X$
is rational. The algorithm is implemented by N. Bruin in Magma, so here
we indicate the main steps only, the actual computations can be carried
out by Magma. We can transform $C_1$ to Weierstrass form
$$
E_1:\quad x^3-(\alpha^2+1)x^2-(\alpha^4+4\alpha^3-4\alpha-5)x+
(2\alpha^4-\alpha^3-4\alpha^2-\alpha+4)=y^2.
$$
The torsion subgroup of $E_1$ consists of two elements. Moreover, the
rank of $E_1$ is two, which is less than the degree of the number field
$K$. Applying elliptic Chabauty (the procedure "Chabauty" of Magma) with
$p=3$, we obtain that $X\in\{1,-1/3\}$.

In case of $C_2$ a similar procedure works. Now the corresponding
elliptic curve $E_2$ is of rank two. Applying elliptic Chabauty this time
with $p=7$, we get that $X=1$, and the lemma follows.
\end{proof}

\begin{lemma}
\label{l25b}
Let $\beta=(1+\sqrt{5})/2$ and put $L=\mathbb{Q}(\beta)$. Then the
only solutions to the equation
\begin{equation}
\label{e25b}
C_3:\quad X^4+(8\beta-12)X^3+(16\beta-30)X^2+(8\beta-12)X+1=Y^2
\end{equation}
in $X\in\mathbb{Q}$, $Y\in L$ are $(X,Y)=(0,\pm 1)$.
\end{lemma}

\begin{proof} The proof is similar to that of Lemma \ref{l25a}.
We can transform $C_3$ to Weierstrass form
$$
E_3:\quad x^3-(\beta-1)x^2-(\beta+2)x+2\beta=y^2.
$$
The torsion group of $E_3$ consists of four points and $(x,y)=(\beta-1,1)$
is a point of infinite order. Applying elliptic Chabauty with $p=13$, we
obtain that $(X,Y)=(0,\pm 1)$ are the only affine points on $C_3$ with
rational first coordinates.
\end{proof}

Now we can give the

\begin{proof}[Proof of Theorem \ref{thm25}]
Suppose that we have a four-term progression of the desired form. Then
by Lemmas \ref{nn2}, \ref{nnn} and the result of Fermat and Euler we
obtain that all the possibilities (except for the ones given in the
statement) are
$$
(l_1,l_2,l_3,l_4)=(2,2,5,5),(5,5,2,2),(2,5,5,2),
$$
$$
(5,2,2,5),(2,2,5,2),(2,5,2,2).
$$
We show that these possibilities cannot occur. Observe that by symmetry
we may assume that we have
$$
(l_1,l_2,l_3,l_4)=(2,2,5,5),(2,5,5,2),(5,2,2,5),(2,2,5,2).
$$
In the first two cases the progression has a sub-progression of the shape
$a^2,b^5,c^5$. Note that here $\gcd(b,c)=1$ and $bc$ is odd. Indeed, if $c$
would be even then we would get $4\mid a^2,c^5$, whence it would follow that
$b$ is even - a contradiction. Taking into consideration the fourth term of
the original progression, a similar argument shows that $b$ is also odd.
Using this subprogression we obtain the equality $2b^5-c^5=a^2$. Putting
$\alpha=\sqrt[5]{2}$ we get the factorization
\begin{equation}
\label{fact}
(\alpha b-c)(\alpha^4b^4+\alpha^3b^3c+\alpha^2b^2c^2+\alpha bc^3+c^4)=a^2
\end{equation}
in $K=\mathbb{Q}(\alpha)$. Note that the class number of $K$ is one,
$\alpha^4,\alpha^3,\alpha^2,\alpha,1$ is an integral basis of $K$,
$\varepsilon_1=\alpha-1$, $\varepsilon_2=\alpha^3+\alpha+1$ provides
a system of fundamental units of $K$ with
$N_{K/\mathbb{Q}}(\varepsilon_1)=N_{K/\mathbb{Q}}(\varepsilon_2)=1$, and
the only roots of unity in $K$ are given by $\pm 1$. A simple calculation
shows that
$$
D:=\gcd(\alpha b-c,\alpha^4 b^4+\alpha^3 b^3c+\alpha^2 b^2c^2+
\alpha bc^3+c^4)\mid\gcd(\alpha b-c,5\alpha bc^3)
$$
in the ring of integers $O_K$ of $K$. Using $\gcd(b,c)=1$ and $2\nmid c$
in $\mathbb{Z}$, we get $D\mid 5$ in $O_K$. Using e.g. Magma, one
can easily check that
$5=(3\alpha^4+4\alpha^3-\alpha^2-6\alpha-3)(\alpha^2+1)^5$, where
$3\alpha^4+4\alpha^3-\alpha^2-6\alpha-3$ is a unit in $K$, and
$\alpha^2+1$ is a prime in $O_K$ with
$N_{K/\mathbb{Q}}(\alpha^2+1)=5$. By the help of these information, we
obtain that
$$
\alpha b-c=(-1)^{k_0}(\alpha-1)^{k_1}(\alpha^3+\alpha+1)^{k_2}
(\alpha^2+1)^{k_3} z^2
$$
with $k_0,k_1,k_2,k_3\in\{0,1\}$ and $z\in O_K$. Taking the norms of
both sides of the above equation, we get that $k_0=k_3=0$. Further, if
$(k_1,k_2)=(0,0),(1,1),(0,1)$ then putting
$z=z_4\alpha^4+z_3\alpha^3+z_2\alpha^2+z_1\alpha+z_0$ with
$z_i\in\mathbb{Z}$ $(i=0,\hdots,4)$ and expanding the right hand side
of the above equation, we get $2\mid b$, which is a contradiction.
(Note that to check this assertion, in case of $(k_1,k_2)=(0,1)$ one
can also use that the coefficients of $\alpha^2$ and $\alpha^3$ on the
left hand side are zero.) Hence we may conclude that $(k_1,k_2)=(1,0)$.
Thus using \eqref{fact} we get that
$$
\alpha^4b^4+\alpha^3b^3c+\alpha^2b^2c^2+\alpha bc^3+c^4=(\alpha-1)y^2
$$
with some $y\in O_K$. Hence after dividing this equation by $c^4$ (which
cannot be zero), we get \eqref{e25a1}, and then a contradiction by Lemma
\ref{l25a}. Hence the first two possibilities for $(l_1,l_2,l_3,l_4)$
are excluded.

Assume next that $(l_1,l_2,l_3,l_4)=(5,2,2,5)$. Then we have
$2x_1^5+x_4^5=3x_2^2$. Using the notation of the previous paragraph, we
can factorize this equation over $K$ to obtain
\begin{equation}
\label{fact2}
(\alpha x_1+x_4)
(\alpha^4x_1^4-\alpha^3x_1^3x_4+\alpha^2x_1^2x_4^2-\alpha x_1x_4^3+x_4^4)
=3x_2^2.
\end{equation}
Observe that the primitivity condition implies that $\gcd(x_1,x_4)=1$,
and $2\nmid x_1x_4$. Hence in the same manner as before
we obtain that the greatest common divisor of the terms on the left hand
side of \eqref{fact2} divides $5$ in $O_K$. Further, a simple calculation
e.g. with Magma yields that
$3=(\alpha+1)(\alpha^4-\alpha^3+\alpha^2-\alpha+1)$, where $\alpha+1$ and
$\alpha^4-\alpha^3+\alpha^2-\alpha+1$ are primes in $O_K$ with
$N_{K/\mathbb{Q}}(\alpha+1)=3$ and
$N_{K/\mathbb{Q}}(\alpha^4-\alpha^3+\alpha^2-\alpha+1)=81$, respectively.
Using these information we can write
$$
\alpha x_1+x_4=(-1)^{k_0}(\alpha-1)^{k_1}(\alpha^3+\alpha+1)^{k_2}
(\alpha+1)^{k_3}(\alpha^4-\alpha^3+\alpha^2-\alpha+1)^{k_4} z^2
$$
with $k_0,k_1,k_2,k_3,k_4\in\{0,1\}$ and $z\in O_K$. Taking the norms of
both sides of the above equation, we get that $k_0=0$ and $k_3=1$.
Observe that $k_4=1$ would imply $3\mid x_1,x_4$. This is a contradiction,
whence we conclude $k_4=0$. Expanding the above equation as previously,
we get that if $(k_1,k_2)=(0,1),(1,0),(1,1)$ then $x_1$ is even, which is
a contradiction again. (To deduce this assertion, when $(k_1,k_2)=(1,1)$
we make use of the fact that the coefficients of $\alpha^3$ and $\alpha^2$
vanish on the left hand side.) So we have $(k_1,k_2)=(0,0)$, which by the
help of \eqref{fact2} implies
$$
\alpha^4x_1^4-\alpha^3x_1^3x_4+\alpha^2x_1^2x_4^2-\alpha x_1x_4^3+x_4^4=
(\alpha^4-\alpha^3+\alpha^2-\alpha+1)y^2
$$
with some $y\in O_K$. However, after dividing this equation by $x_1^4$
(which is certainly non-zero), we get \eqref{e25a2}, and then a
contradiction by Lemma \ref{l25a}.

Finally, suppose that $(l_1,l_2,l_3,l_4)=(2,2,5,2)$. Using the identity
$x_2^2+x_4^2=2x_3^5$, e.g. by the help of a result of Pink and Tengely
\cite{PT} we obtain
$$
x_2=u^5-5u^4v-10u^3v^2+10u^2v^3+5uv^4-v^5
$$
and
$$
x_4=u^5+5u^4v-10u^3v^2-10u^2v^3+5uv^4+v^5
$$
with some coprime integers $u,v$. Then the identity
$3x_2^2-x_4^2=2x_1^2$ implies
\begin{equation}
\label{fact3}
(u^2-4uv+v^2)f(u,v)=x_1^2
\end{equation}
where
$$
f(u,v)=u^8-16u^7v-60u^6v^2+16u^5v^3+134u^4v^4+
$$
$$
+16u^3v^5-60u^2v^6-16uv^7+v^8.
$$
A simple calculation shows that the common prime divisors of the
terms at the left hand side belong to the set $\{2,5\}$. However,
$2\mid x_1$ would imply $4\mid x_1^2,x_3^5$, which would violate the
primitivity condition. Further, if $5\mid x_1$ then looking at the
progression modulo $5$ and using that by the primitivity condition
$x_2^2\equiv x_4^2\equiv \pm 1\pmod{5}$ should be valid, we get a
contradiction. Hence the above two terms are coprime, which yields
that
$$
f(u,v)=w^2
$$
holds with some $w\in\mathbb{Z}$. (Note that a simple consideration
modulo $4$ shows that $f(u,v)=-w^2$ is impossible.) Let
$\beta=(1+\sqrt{5})/2$, and put $L=\mathbb{Q}(\beta)$. As is
well-known, the class number of $L$ is one, $\beta,1$ is an integral
basis of $L$, $\beta$ is a fundamental unit of $L$ with
$N_{L/\mathbb{Q}}(\beta)=1$, and the only roots of unity in $L$ are
given by $\pm 1$. A simple calculation shows that
$$
f(u,v)=g(u,v)h(u,v)
$$
with
$$
g(u,v)=u^4+(8\beta-12)u^3v+(16\beta-30)u^2v^2+(8\beta-12)uv^3+v^4
$$
and
$$
h(u,v)=u^4+(-8\beta-4)u^3v+(-16\beta-14)u^2v^2+(-8\beta-4)uv^3+v^4.
$$
Further, $\gcd(6,x_1)=1$ by the primitivity of the progression, and one
can easily check modulo $5$ that $5\mid x_1$ is also impossible. Hence
we conclude that $g(u,v)$ and $h(u,v)$ are coprime in the ring $O_L$
of integers of $L$. Thus we have
$$
g(u,v)=(-1)^{k_0} \beta^{k_1} z^2
$$
with some $k_0,k_1\in\{0,1\}$ and $z\in O_L$. Note that as
$2\nmid x_1$, equation \eqref{fact3} implies that exactly one of $u,v$
is even. Hence a simple calculation modulo $4$ shows that the only
possibility for the exponents in the previous equation is
$k_0=k_1=0$. However, then after dividing the equation with $v^4$
(which cannot be zero), we get \eqref{e25b}, and then a contradiction
by Lemma \ref{l25b}.

There remains to show that a four-term progression with exponents
$(l_1,l_2,l_3,l_4)=(2,2,2,5)$ or $(5,2,2,2)$ cannot be extended to a
five-term one. By symmetry it is sufficient to deal with the first case.
If we insert a square or a fifth power after the progression, then the
last four terms yield a progression which has been already excluded.
Writing a fifth power, say $x_0^5$ in front of the progression would
give rise to the identity $x_0^5+x_4^5=2x_2^2$, which leads to a
contradiction by Lemma \ref{nn2}. Finally, putting a square in front
of the progression is impossible by the already mentioned result of
Fermat and Euler. 
\end{proof}

\section{Acknowledgement}
The research of the first author was supported in part by the National
Office for Research and Technology.

\end{document}